\documentclass[11pt,  reqno]
{amsart}
\usepackage{amsmath,amssymb,amscd,amsthm,amsxtra, esint}

\allowdisplaybreaks[2]

\newtheorem{theorem}{Theorem} [section]

\newtheorem{lemma}[theorem]{Lemma}
\newtheorem{proposition}[theorem]{Proposition}
\newtheorem{remark}[theorem]{Remark}

\newtheorem{definition}[theorem]{Definition}

\DeclareMathOperator*{\supp}{supp}

\newcommand{\I}{\hspace{0.5mm}\text{I}\hspace{0.5mm}}
\newcommand{\noi}{\noindent}
\newcommand{\Z}{\mathbb{Z}}
\newcommand{\R}{\mathbb{R}}
\newcommand{\C}{\mathbb{C}}
\newcommand{\T}{\mathbb{T}}

\let\Re=\undefined\DeclareMathOperator*{\Re}{Re}
\let\Im=\undefined\DeclareMathOperator*{\Im}{Im}

\renewcommand{\L}{\mathcal{L}}

\newcommand{\N}{\mathcal{N}}

\newcommand{\D}{\mathcal{D}}

\newcommand{\A}{\mathcal{A}}

\newcommand{\al}{\alpha}

\newcommand{\dl}{\delta}

\newcommand{\eps}{\varepsilon}

\newcommand{\G}{\Gamma}
\newcommand{\ld}{\lambda}

\newcommand{\s}{\sigma}

\newcommand{\ft}{\widehat}

\newcommand{\cj}{\overline}
\newcommand{\dx}{\partial_x}
\newcommand{\dt}{\partial_t}

\renewcommand{\l}{\ell}
\renewcommand{\o}{\omega}

\newcommand{\les}{\lesssim}

\newcommand{\jb}[1]
{\langle #1 \rangle}

\newcommand{\too}{\longrightarrow}

\newcommand{\oo}{\pmb{\o}}
\newcommand{\n}{\pmb{n}}

\numberwithin{equation}{section}
\numberwithin{theorem}{section}

\begin{document}

\baselineskip = 14pt

%\date{\today}
%%

\title[NLS with almost periodic initial data]
{On nonlinear Schr\"odinger equations
with almost periodic initial data}

\author{Tadahiro Oh}

\address{
Tadahiro Oh\\
School of Mathematics\\
The University of Edinburgh,
and The Maxwell Institute for the Mathematical Sciences\\
James Clerk Maxwell Building\\
The King's Buildings\\
Peter Guthrie Tait Road\\
Edinburgh\\ EH9 3FD\\ United Kingdom} 

\email{hiro.oh@ed.ac.uk}



\begin{abstract}

We consider the Cauchy problem
of nonlinear Schr\"odinger equations (NLS)
with almost periodic functions as initial data.
We first prove that, given
%an integral basis
a frequency set $\pmb{\o} =\{ \o_j\}_{j = 1}^\infty$,
NLS is  local well-posed
in the algebra  $\A_{\pmb{\o}}(\R)$ of almost periodic functions
with absolutely convergent Fourier series.
Then, we prove a finite time blowup result
for  NLS with a nonlinearity $|u|^p$, $p \in 2\mathbb{N}$.
This %elementary argument presents 
provides the first instance of finite time blowup solutions
to NLS with generic almost periodic initial data.
\end{abstract}

\maketitle

\section{Introduction}

We consider the Cauchy problem
of the following nonlinear Schr\"odinger equation (NLS)
with an algebraic power-type nonlinearity:
\begin{align}
\begin{cases}
i \dt u + \dx^2 u = \N(u), \\
u|_{t = 0} = f,
\end{cases}
\quad (t, x) \in \R\times \R,
\label{NLS1}
\end{align}

\noi
where the nonlinearity is given by $\N(u) = \N_p(u, \cj u) = u^k \cj u^{p-k}$, $0 \leq k \leq p$,
$p\in \mathbb{N}$, $k \in \mathbb{N} \cup\{0\}$.
For example, this includes the standard
power nonlinearity $|u|^{p-1} u$
and a nonlinearity $|u|^p$ without gauge invariance.

The Cauchy problem \eqref{NLS1} has been studied extensively
in terms of the usual
 Sobolev spaces $H^s(\R)$ on the real line and
 the Sobolev spaces $H^s_{\text{per}}(\R) \simeq H^s(\T)$
 of periodic functions (of a fixed period) on $\R$.
See \cite{CAZ, TAO} for the references therein.
Our main interest in this paper is to study the Cauchy problem \eqref{NLS1}
with {\it almost periodic} functions as initial data.

\begin{definition}\rm
We say that a complex-valued function $f$ on $\R$ is  {\it almost periodic},
if it is continuous and,  for every $\eps > 0$, there exists  $L = L(\eps, f)>0$
such that every interval of length $L$ on $\R$ contains
 a number $\tau$ such that
\[ \sup_{x \in \R} |f(x-\tau) - f(x)|< \eps.\]

\noi
We use $AP(\R)$ to denote the space of almost periodic functions on $\R$.
\end{definition}

\noi
The study of almost periodic functions was initiated by Bohr \cite{Bohr}.
In the following, we briefly go over the basic properties
of almost periodic functions.
See Besicovitch \cite{B}, Corduneanu \cite{C}, and Katznelson \cite{K}
for more on the subject.
Let us first state
several equivalent
characterizations
for almost periodic functions.

\begin{definition}\rm

\noi
(i) We say that a function $f$ on $\R$ has the {\it approximation property},
if it can be uniformly approximated by trigonometric polynomials.
More precisely, given any $\eps > 0$,
there exists a trigonometric polynomial $P_\eps(x)$
such that
\[ \sup_{x \in \R} |f(x) - P_\eps(x)|< \eps.\]

\noi
(ii) We say that a continuous function on $\R$ is {\it normal}
if, given any $\{x_n\}_{n = 1}^\infty \subset \R$,
the collection $\{ f(\,\cdot  + x_n)\}_{n = 1}^\infty$ is precompact in $L^\infty(\R)$.
Namely, there exists a subsequence $\{ f(\,\cdot + x_{n_j})\}_{j = 1}^\infty$
uniformly convergent on $\R$.
\smallskip

\end{definition}

\noi
An important fact is that   the set of almost periodic functions,
 the set of functions with the approximation property,
 and  the set of normal functions all coincide.
Hence, we freely use any of these three characterizations
in the following.
We also point out
 that these three  notions can be extended to
Banach-space valued functions
and that they are also equivalent in the Banach space setting.
Given a Banach space $X$,
we use $AP(\R; X)$ to denote the space of almost periodic functions on $\R$
with values in $X$.

It is known that
$AP(\R)$ is a closed subalgebra of $L^\infty(\R)$
and that
almost periodic functions are uniformly continuous.
 Given $f \in AP(\R)$,
 we can define the so-called
 {\it mean value} $M(f)$  of $f$
 by
\begin{align}
	 M(f) := \lim_{L\to \infty} \frac{1}{2L}\int_{-L}^L f(x) dx,
\label{AP1}
\end{align}

\noi
where the limit on the right-hand side of \eqref{AP1}
always exists if $f \in AP(\R)$.
Given  $f \in AP(\R)$, we  define
the $\L^2$-norm by the mean value of $|f|^2$:
\begin{align}
\|f\|_{\L^2} := \lim_{L\to \infty} \bigg(\frac{1}{2L}\int_{-L}^L |f(x)|^2 dx\bigg)^\frac{1}{2}.
\label{AP2}
\end{align}

\noi
Note that the limit on the right-hand side of \eqref{AP2} exists
since the algebra property of $AP(\R)$
states that  $|f|^2$ is almost periodic, if $f \in AP(\R)$.
We have the following lemma.

\begin{lemma}[Lemma on p.\,177 in \cite{K}]\label{LEM:AP1}
Let $f \in AP(\R)$ such that $f \geq 0 $ on $\R$.
If $f$ is not identically equally to 0,
then $M(f)>0$.
In particular, the $\L^2$-norm, defined in \eqref{AP2},  of a function $f \in AP(\R)$
is 0 if and only if $f \equiv 0$.
Hence, it  is indeed a norm on $AP(\R)$.

\end{lemma}

\noi
Note that the claim of Lemma \ref{LEM:AP1}
does not hold in general, if $f \notin AP(\R)$.
For example, we have $M(f) = 0$ for any bounded function $f$ with a compact support.

Next, we  define an inner product $\jb{\cdot, \cdot}_{\L^2}$ on $AP(\R)$ by
\begin{align}
\jb{f, g}_{\L^2} := M(f\cj{g}) = \lim_{L\to \infty} \frac{1}{2L}\int_{-L}^L f(x) \cj{g}(x) dx.
\label{AP3}
\end{align}

\noi	
for $f, g \in AP(\R)$.
This inner product  is well defined for $f, g \in AP(\R)$,
since $f\cj{g}$ is also in $AP(\R)$.
Moreover, it induces the $\L^2$-norm defined in \eqref{AP2}.
Therefore,
under the inner product $\jb{\cdot, \cdot}_{\L^2}$,
the space $AP(\R)$ of almost periodic functions
becomes a pre-Hilbert space (missing completeness).\footnote{In this paper, we only consider almost periodic function
in Bohr's sense.  There are, however, notions of different classes of generalized almost periodic functions due to Stepanov,
Weyl, and Besicovitch. The corresponding spaces are denoted by $S^p, W^p$ and $B^p$, respectively.
Then, we have $AP(\R) \subset S^p \subset W^p \subset B^p$, $p \geq 1$.
Moreover,  it is known that $B^2$ is complete with respect to the
$\mathcal{L}^2$-norm defined in \eqref{AP2}.}

In this pre-Hilbert space, the complex exponentials
$\{ e^{ i \o x}\}_{\o \in \R}$ form an orthonormal family.
We now define the Fourier coefficient of $f \in AP(\R)$ by
\begin{align}
\ft f (\o) = \jb {f, e^{ i \o x}}_{\L^2} = M(f e^{- i \o x}).
\label{AP2a}
\end{align}

\noi
By Bessel's inequality, we have
\[ \sum_{\o \in \R} |\ft f (\o)|^2 \leq \|f\|_{\L^2}^2 < \infty.\]

\noi
In particular, this implies that $\ft f(\o) = 0$ except
for  countable many values of $\o$'s.
Given $f \in AP(\R)$, we define its frequency set $\s(f)$
by
$\s(f) : = \{ \o \in \R:\, \ft f(\o) \ne 0\}$
and write
\begin{equation}
 f (x) \sim \sum_{\o \in \s(f)} \ft f (\o) e^{i \o x},
\label{AP3c}
 \end{equation}

\noi
where the right-hand side is the Fourier
series associated to $f \in AP(\R)$.
It is known that the orthonormal family
$\{ e^{ i \o x}\}_{\o \in \R}$ is complete
in the sense that two distinct almost periodic functions
have distinct Fourier series.
Moreover, we have the Parseval's identity:
\begin{align}
  \|f\|_{\L^2} = \bigg(\sum_{\o \in \R} |\ft f (\o)|^2 \bigg)^\frac{1}{2}
\label{AP3b}
  \end{align}

\noi
for $f \in AP(\R)$.
Regarding the actual convergence of the Fourier series
to an almost periodic function, we have the following lemma.

\begin{lemma}[Theorem 1.20 in \cite{C}] \label{LEM:AP2}
Let $f \in AP(\R)$.
If the Fourier series associated to $f$ converges uniformly,
then it converges to $f$.  Namely, we have
\begin{align}
 f (x) = \sum_{\o \in \s(f)} \ft f (\o) e^{i \o x}.
\label{AP3a}
 \end{align}
\end{lemma}

Given $\oo = \{\o_j\}_{j = 1}^\infty
\in \R^{\mathbb N}$,  we
say that $\oo$ is linear independent
if any relation of the form:
\[ \sum_{j = 1}^N r_j \o_j =0, \quad r_j  \in \mathbb Q, \]

\noi
implies that $r_j = 0$, $j = 1, \dots, N$.
Associated to this notion of linear independence,
there is an important criterion
on the convergence of the Fourier series
to a given almost periodic function.

\begin{lemma}[Theorem 1.25 in \cite{C}] \label{LEM:AP2a}
Let  $\oo = \{\o_j\}_{j = 1}^\infty
\in \R^{\mathbb N}$ be linearly independent.
Suppose that  $f \in AP(\R)$ satisfies $\s(f) \subset \oo$.
Then,  the Fourier series associated to $f$ converges uniformly.
In particular, \eqref{AP3a} holds.
\end{lemma}

Given a set $S$ of real numbers,
we say that a linearly independent
set $\oo = \{\o_j\}_{j = 1}^\infty$
is a basis for the set $S$, if every element in $ S$
can be represented as
a finite linear combination of elements in $\oo$
with rational coefficients.
Given $f \in AP(\R)$,
we say that a linearly independent set  $\pmb{\o} = \{\o_j\}_{j = 1}^N$,
allowing the case $N = \infty$,
is a basis of $f$, if
it is a  basis of
the frequency set
$\s(f)$ of $f$.
Lemma 1.14 in \cite{C} guarantees existence of a basis of $f \in AP(\R)$.
We say that a basis $\pmb{\o} = \{\o_j\}_{j = 1}^N$ of $f$ is an integral basis
if any element in the frequency set $\s(f)$ can be written as a finite
linear combination of elements in $\oo$
with integer coefficients.\footnote{Obviously, an almost periodic function is periodic
if and only if it has an integral basis consisting of a single element $\o \in \R$.}
If there exists a finite integral basis of $f$, i.e.~$N < \infty$, then we say that the function $f$
is {\it quasi-periodic}.
In this paper, we consider generic
almost periodic functions, i.e.~$N= \infty$,
but the corresponding results also hold for
quasi-periodic functions, i.e.~$N< \infty$.

Fix  $\pmb{\o} = \{\o_j\}_{j = 1}^\infty \in \R^{\mathbb N}$.
We consider functions $f \in AP(\R)$
with $\s (f) \subset \oo \cdot \Z^\mathbb{N}$
of the form:
\begin{align}
	 f(x) \sim \sum_{\n \in \Z^\mathbb{N}} \ft f(\pmb{\o}\cdot \pmb{n}) e^{i (\oo \cdot \n)x},
\label{AP4}
\end{align}

\noi
where $\n = \{ n_j\}_{j = 1}^\infty \in \Z^\mathbb{N}$.
%Note that, if $\oo$ is linearly independent, the it is an integral basis of $f$.
%We, however, do not assume such linear independence of $\o$.
We  define the algebra $\A_{\oo}(\R)$ by
\begin{align*}
\A_{\oo}(\R) = \big\{ f \in AP(\R):\,
f \text{ is of the form } \eqref{AP4} \text{ and }
 \| f\|_{\A_{\oo}(\R)} < \infty\big\},
%\label{AP4c}
 \end{align*}

\noi
where the $\A_{\oo}(\R)$-norm is given by
\begin{align*}
\| f\|_{\A_{\oo}(\R)} = \| \ft f(\oo\cdot \n)\|_{\l^1_{\n}(\Z^\mathbb{N})}.
\end{align*}

\noi
See Lemma \ref{LEM:AP3} below for some properties of $\A_{\oo}(\R)$.

\begin{remark}\label{REM:FS} \rm
Note that, if $\oo$ is linearly independent, then it is an integral basis of $f$.
%We, however, do not assume such linear independence of $\oo$.
If $\oo$ is not linearly independent,
then, we may have $\oo\cdot \n_1 = \oo \cdot \n_2$ for some $\n_1 \ne \n_2$.
Namely, $\ft f (\oo\cdot \n)$ in \eqref{AP4} may not represent a Fourier coefficient
of $f$ defined in \eqref{AP2a} and \eqref{AP3c}.
In this case, %Given $\al \in \R$,
the Fourier coefficient $\ft f(\al)$, $\al \in \R$, is given by
$\ft f(\al) = \sum_{\oo\cdot\n = \al}\ft f (\oo\cdot \n)$.
In the following (for example, see Lemma \ref{AP3} below),
we proceed, assuming that $\oo$ is linearly independent.
We point out that the results also hold even when $\oo$ is not linearly independent.
It suffices to note that
 the definition of $\A_{\oo}(\R)$ guarantees that
the Fourier coefficients $\ft f (\al)$ of $f \in \A_{\oo}(\R)$
is absolutely summable.

\end{remark}

\smallskip

We are now ready to state our first result.

\begin{theorem}\label{THM:LWP}
Let $p \in \mathbb{N}$.
Fix $\oo = \{\o_j\}_{j = 1}^\infty \in \R^{\mathbb N}$.
Then,
NLS \eqref{NLS1} is locally well-posed in $\A_{\oo}(\R)$.
More precisely, given
$f \in \A_{\oo}(\R)$, there exist $T = T(\|f\|_{\A_{\oo}(\R)})>0$
and unique $u \in C([-T, T]; \A_{\oo}(\R))$
satisfying the following Duhamel formulation of \eqref{NLS1}:
\begin{align}
u(t) = S(t) f -i \int_0^t S(t-t') \N(u)(t') dt',
\label{NLS3}
\end{align}

\noi
where  $S(t) = e^{i t\dx^2}$.
Moreover, the solution map $: f \in \A_{\oo}(\R) \mapsto u(t) \in \A_{\oo}(\R)$
is locally Lipschitz continuous.

\end{theorem}
	
\noi
Our solution $u(t)$ lies in $\A_{\oo}(\R)$ for all $t \in [-T, T]$.
In particular, $u(t)$ is almost periodic in $x$ for all $t \in [-T, T]$.
Moreover, it satisfies \eqref{NLS1} in the distributional sense.
See Lemmata \ref{LEM:linear1} and \ref{LEM:linear2} below.

In Section \ref{SEC:LWP}, we
define the meaning of the linear propagator
$S(t) = e^{i t\dx^2}$ in the almost periodic setting
and discuss different properties of solutions
to the homogeneous and nonhomogeneous linear Schr\"odinger equations
in the almost periodic setting.
Then, we  present the proof of Theorem \ref{THM:LWP},
 based on a simple fixed point argument.
Since our approach makes use of the Fourier coefficients
of functions in $\A_{\oo}(\R)$,
it is essential that the Fourier series associated to a function in $\A_{\oo}(\R)$
actually converges to it.
See Lemma \ref{LEM:AP3}.

\begin{remark} \label{REM:LWP} \rm
Previously, Tsugawa \cite{Tsugawa} proved
local well-posedness of the Korteweg-de Vries equation (KdV) on $\R$:
\begin{align}
\dt u + \dx^3 u = u  \dx u
\label{KdV}
\end{align}

\noi
with quasi-periodic initial data
under some regularity condition.	
For fixed $\oo = \{\o_j\}_{j = 1}^N \in \R^N$ for some finite $N \in \mathbb N$,
consider a quasi-periodic function $f$ of the form:
\begin{align}
	 f(x) = \sum_{\n \in \Z^N} \ft f(\pmb{\o}\cdot \pmb{n}) e^{i (\oo \cdot \n)x}.
\label{AP4a}
\end{align}

\noi
Defining a Sobolev-type space\footnote{This Sobolev-type
space  $\mathcal{H}^{\pmb{s}}(\R)$ is basically
the space $G^{\pmb{s}, 0}$ defined in \cite{Tsugawa}.}
 $\mathcal{H}^{\pmb{s}}_{\oo}(\R)$
for  $\pmb{s} = \{s_j\}_{j = 1}^N \in \R^{N}$ by the norm
 \begin{align}
  \|f\|_{\mathcal{H}^{\pmb{s}}_{\oo}(\R)}
 :=  \|    \jb{\n}^{\pmb{s}}\ft f (\oo\cdot \n) \|_{\l^2_{\n}(\Z^N)},
 \quad
 \jb{\n}^{\pmb{s}} : =  \prod_{j = 1}^N ( 1+ |n_j|^2)^\frac{s_j}{2},
 \label{AP4b}
 \end{align}

\noi
it follows from
Lemma 2.2 (i) in \cite{Tsugawa}
that NLS \eqref{NLS1} is locally well-posed
in $\mathcal{H}^{\pmb{s}}_{\oo}(\R)$
as long as $\min(s_1, \dots, s_N) > \frac{1}{2}$.
In this case, we have $\A_{\oo}(\R) \supset
\mathcal{H}^{\pmb{s}}_{\oo}(\R)$ by Cauchy-Schwarz inequality,
and thus Theorem \ref{THM:LWP} extends
this local well-posedness result of \eqref{NLS1}
with quasi-periodic initial data
implied by Lemma 2.2 (i) in  \cite{Tsugawa}.
It is not clear if there is a natural way to define
a Sobolev-type space analogous to \eqref{AP4b}
in the almost periodic setting,
which guarantees that every function in the space
can be represented by its Fourier series.

\end{remark}

In view of Theorem \ref{THM:LWP},
it is natural to consider the global-in-time behavior
of solutions to \eqref{NLS1}.
This is, however, an extremely difficult question in general.
Consider the following  NLS with the standard power nonlinearity:
\begin{align}
i \dt u + \dx^2 u = \pm |u|^{p-1} u.
\label{PNLS}
\end{align}

\noi
A standard approach to construct global-in-time solutions
is to use conservation laws.
There are several (formal) conservation laws for \eqref{PNLS}, including
the mass conservation $Q(u) : = \|u\|_{\L^2}^2 = M(|u|^2)$
and the `Hamiltonian' conservation:
\begin{align*}
H(u) = \frac{1}{2} M(|\dx u|^2) \pm  \frac{1}{p+1} M(|u|^{p+1}).
\end{align*}

On the one hand,
in order to make use of the mass conservation
in constructing global-in-time solutions,
one needs to prove local well-posedness in
$AP(\R)$ endowed with the $\L^2$-norm.
This seems to be beyond our current technology
in the quasi- and almost periodic setting
due to the lack of Strichartz estimates.
Note that
while every function $f \in AP(\R)$ has a finite $\L^2$-norm,
(i) $AP(\R)$ is not complete with respect to the $\mathcal L^2$-norm
and (ii)
its Fourier series does not necessarily converges to $f$.
Hence, in proceeding with Fourier analytic approach,
it seems that one needs to work in a subclass,
where functions are actually represented by  their Fourier series.
For example, see Lemma \ref{LEM:AP2a} above.

On the other hand,
assuming that $u$ is of the form \eqref{AP4}, we formally have
\[ M(|\dx u|^2) = \sum_{\n \in \Z^{\mathbb{N}} }(\oo\cdot\n)^2 |\ft u (\oo\cdot\n)|^2.\]

\noi
In general, $(\oo\cdot\n)^2$ can be arbitrarily close to 0
and thus $ M(|\dx u|^2)$ (and hence the Hamiltonian) is not strong enough to control
relevant norms for iterating a local argument.

There are, however, several known
global existence results for  cubic NLS,  \eqref{PNLS} with $p = 3$,
and KdV in the almost periodic and quasi-periodic setting.
Note that the following results rely heavily on the inverse spectral method and
on the complete integrability of the equations.\footnote{
There is a special subclass of almost periodic functions called {\it limit periodic} functions,
consisting of  uniform limits of periodic functions.
In our recent paper \cite{O2}, we proved global well-posedness of the 
defocusing NLS with nonlinearity $|u|^{2k}u$, $k \in \mathbb{N}$, 
with limit periodic functions as initial data
under some regularity assumption.
In particular, our proof does not rely on the completely integrability, even when $k = 1$.
Moreover, when $k \geq 2$,  
it provides
the first instance of global existence for the defocusing NLS with (a subclass of) almost periodic initial data that are not quasi-periodic.}
Egorova \cite{E}
and Boutet de Monvel-Egorova \cite{BE}
constructed global-in-time solutions
to KdV and  cubic NLS with almost periodic initial data,
assuming some conditions,
including Cantor-like spectra for
the  corresponding
Schr\"odinger operator (for KdV)
and
Dirac operator (for cubic NLS).
In particular,
the class of almost periodic initial data in \cite{E, BE}
includes
almost periodic functions $f$
that can be approximated by periodic functions $f_n$
of growing periods $\al_n \to \infty$
in a local Sobolev norm: $\sup_{x\in \R} \| \cdot \|_{H^s([x, x+1])}$
with $s \geq 4$ for KdV and $s \geq 3$ for cubic NLS.\footnote{Note that when $s = 0$, this local Sobolev norm
corresponds to  Stepanov's $S^2$-norm used for Stepanov's generalized almost periodic functions.} 
Moreover,  convergence of $f_n$ to $f$ in this local Sobolev norm is assumed to be exponentially fast.
It is worthwhile to mention that the solutions constructed in \cite{E, BE}
are almost periodic in both $t$ and $x$.
There is also a recent global well-posedness result of KdV
with quasi-periodic initial data
by Damanik-Goldstein \cite{DG}.
Their result states that if  the Fourier coefficient $\ft f(\oo\cdot \n)$
of a `small' quasi-periodic initial condition $f$ of the form \eqref{AP4a}
decays exponentially fast (in $\n$)
and a Diophantine condition on $\oo$ is satisfied,
then there exists a unique global solution
whose Fourier coefficient also decays exponentially fast
(with a slightly worse constant).

Another approach for constructing global solutions
is to consider the problem for small initial data.
Indeed, in the usual setting on $\R^d$, i.e.~assuming that  functions belong to the usual
Lebesgue spaces $L^q(\R^d)$ or
Sobolev spaces $H^s(\R^d)$,
we have small data global well-posedness and scattering
for NLS \eqref{NLS1} on $\R^d$,
for example, for
 $p_s < p < 1+ \frac{4}{d-2}$,
where  $p_s$ is the Strauss exponent  given by
\begin{equation*}
p_s = \frac{d+2 +\sqrt{d^2 + 12 d + 4}}{2d}.
\end{equation*}

\noi
The proof of this result relies on the decay of linear solutions on $\R^d$.
See \cite{CAZ}.
On the contrary,
there is no such decay of linear solutions in the almost periodic setting.
Hence, there seems to be no natural adaptation of small data global existence theory
to the almost periodic setting.

Next, let us discuss finite time blowup solutions to
\eqref{NLS1} in the almost periodic setting.
Since a periodic function is in particular almost periodic,
known results on finite time blowup solutions in the periodic setting such as \cite{OT, O}
provide instances of finite time blowup results
in the almost periodic setting (where initial data are periodic).
There seems to be, however, no known result on finite time blowup solutions
in a generic (i.e.~non-periodic) almost periodic setting.

In the following,
we consider the Cauchy problem
of the following NLS:
\begin{align}
\begin{cases}
i \dt u + \dx^2 u = \ld |u|^p, \\
u|_{t = 0} = f,
\end{cases}
\quad (t, x) \in \R\times \R,
\label{ZNLS1}
\end{align}

\noi
for $p \in 2 \mathbb{N}$ and $\ld \in \C$.
Then, we have the following result on finite time blowup solutions
in a generic almost periodic setting.

\begin{theorem}\label{THM:blowup}
Let $p\in 2\mathbb{N}$.
Fix $\oo = \{\o_j\}_{j = 1}^\infty \in \R^{\mathbb N}$.
Let $u \in C(I; \A_{\oo}(\R))$ be the solution to \eqref{ZNLS1}
with $u|_{t = 0} = f\in \A_{\oo}(\R)$,
where $I = (-T_-,T_+)$ is the maximal time interval of existence,
containing $t = 0$.
 Suppose that $\ld \in \C$
and the mean value $M(f)$ of $f$, defined in \eqref{AP1},  satisfy
\begin{equation}\label{sign0}
\Re\ld \cdot \Im M(f) \ne 0 \quad \text{  or } \quad \Im \ld \cdot \Re M(f)  \ne 0.
\end{equation}

\noi
Then, we have $\min (T_-, T_+) < \infty$.
Namely, the solution $u$ blows up in a finite time,
either forward or backward in time.
More precisely, if \eqref{sign0} holds, then we have one of the following scenarios.

\smallskip

\noi
\textup{(i)}
Suppose that  we have
\begin{equation}\label{sign1}
\Re\ld \cdot \Im M(f) < 0 \quad \text{  or } \quad \Im \ld \cdot \Re M(f)  > 0.
\end{equation}

\noi
Then, the forward maximal time $T_+$ of existence of the solution $u$ is finite
and we have
$ \liminf_{t \nearrow T_+}\|u(t)\|_{\A_{\oo}(\R)} = \infty.$

\smallskip

\noi
\textup{(ii)}
Suppose that  we have
\begin{equation}\label{sign2}
\Re\ld \cdot \Im M(f) > 0 \quad \text{  or } \quad \Im \ld \cdot \Re M(f)  < 0.
\end{equation}

\noi
Then, the backward maximal time $T_-$ of existence of the solution $u$ is finite
and we have
$ \liminf_{t \searrow -T_-}\|u(t)\|_{\A_{\oo}(\R)} = \infty.$

\end{theorem}

\noi
Previously, Ikeda-Wakasugi \cite{IW} and the author \cite{O}
obtained analogous results for \eqref{ZNLS1} on $\R^d$ (with $1 < p \leq \frac{2}{d}$
for initial data in $L^2(\R^d)$)
and  on $\T^d$ (with $p \in 2\mathbb{N}$).
Theorem \ref{THM:blowup}
can be viewed as an extension of the periodic result
in \cite{O} to the generic almost periodic setting.
In particular,
note that
  Theorem \ref{THM:blowup}
holds (i) even for small initial  data
and (ii) even above the Strauss exponent, i.e.~$p > p_s$,
provided that \eqref{sign0} is satisfied.
This is a sharp contrast
with  the usual Euclidean setting on $\R$,
where we have small data global well-posedness when $p > p_s$.
The proof of Theorem \ref{THM:blowup}  follows the basic lines in \cite{IW, O}
but we need to proceed more carefully  due to the almost periodic setting.
As in the proof of Theorem \ref{THM:LWP},
we make essential use of properties of functions in $\A_{\oo}(\R)$.
We present the proof of Theorem \ref{THM:blowup}
in Section \ref{SEC:blowup}.

\begin{remark}\label{REM:bound1}\rm

Suppose that \eqref{sign0} is satisfied.
Then,
it follows from  Theorem \ref{THM:blowup}
that any  solution to \eqref{ZNLS1} on $[0, \infty)$ or $(-\infty, 0]$
must satisfy
a global space-time bound.
For example, if $u$ is a solution to \eqref{ZNLS1} on $[0, \infty)$,
then we have
$\Re\ld \cdot \Im M(f) > 0$ and $ \Im \ld \cdot \Re M(f)  < 0$.
Then, we have
\begin{align*}
 \int_0^\infty M(|u(t)|^p )   dt  <\infty.
\end{align*}

\noi
Hence, in view of   Lemma \ref{LEM:AP1},
 any global solution on $[0, \infty)$ must
go to 0 as $t\to \infty$ in some averaged sense.
See Remark \ref{REM:bound2} for the proof.

\end{remark}

\begin{remark}
\rm
The notion of  almost periodic functions
can be extended to higher dimensions.\footnote{
Note that almost periodic functions on $\R^d$ are  almost
periodic in each variable, but  the converse is not true.}
One may extend
the results in this paper
to the higher dimensional setting.
We, however, focus on the one-dimensional case
for simplicity of the presentation.
\end{remark}

\section{Local well-posedness}\label{SEC:LWP}

In this section, we present the proof of Theorem \ref{THM:LWP}.
Fix  $\oo = \{\o_j\}_{j = 1}^\infty \in \R^{\mathbb N}$ in the following.

\subsection{On the function space $\A_{\oo}(\R)$}
We first go over some important properties of
the space $\A_{\oo}(\R)$
of almost periodic functions
with a common integral basis  $\oo$.

\begin{lemma}\label{LEM:AP3}
Given $f \in \A_{\oo}(\R)$, we have
\begin{align}
	 f(x) = \sum_{\n \in \Z^\mathbb{N}} \ft f(\pmb{\o}\cdot \pmb{n}) e^{i (\oo \cdot \n)x}.
\label{AP5}
\end{align}

\noi
Namely, $f$ is given by its Fourier series.\footnote{If $\oo$ is not linearly independent,
the right-hand side of \eqref{AP5} may not represent the Fourier series associated
to $f$.
Nontheless, the claim in Lemma \ref{LEM:AP3} holds.
 See Remark \ref{REM:FS}.}
Moreover,  $\A_{\oo}(\R)$ is a Banach algebra.
\end{lemma}
	
\begin{proof}
Let $\{ r_j\}_{j = 1}^\infty$
be an enumeration of
 $\Z^{\mathbb N}$.
For $N \in \mathbb N$, we set
$B_N = \{r_j \}_{j = 1}^N$.
Given $f \in \A_{\oo}(\R)$,
 define a trigonometric polynomial $f_N$ by
\begin{align}
	 f_N(x) = \sum_{\n \in B_N } \ft f(\pmb{\o}\cdot \pmb{n}) e^{i (\oo \cdot \n)x},
\label{AP6}
\end{align}

\noi
Then, by the Weierstrass $M$-test, $f_N$ converges uniformly.
Hence,  by Lemma \ref{LEM:AP2}, we obtain \eqref{AP5}.
Note that $f_N$ also converges to $f$ in $\A_{\oo}(\R)$.

Given a Cauchy sequence  $\{ f_k\}_{k = 1}^\infty$ in  $\A_{\oo}(\R)$,
it follows from  the completeness of $\l^1$,
that $f_k$ converges to some function $f$ defined by the Fourier series \eqref{AP5}
with respect to the $\A_{\oo}(\R)$-norm.
We need to show that $f \in AP(\R)$.
Since $\{ f_k\}_{k = 1}^\infty$ is Cauchy with respect to the $\A_{\oo}(\R)$-norm,
then it is Cauchy in $L^\infty(\R)$.
Noting  that $AP(\R)$ is closed with respect to the $L^\infty$-norm,
it follows that
$f_k$  converges some function in $AP(\R)$.
Hence, by uniqueness of a limit, we conclude that $f \in AP(\R)$.
This proves completeness of $\A_{\oo}(\R)$.

Lastly,  given $f, g \in \A_{\oo}(\R)$, we have $f g \in AP(\R)$.
Moreover, we have $\s (fg) \subset \oo \cdot \Z^\mathbb{N}$
and
\begin{align*}
 \ft{fg}(\oo\cdot \n ) =
\sum_{\pmb{m} \in \Z^\mathbb{N}} \ft f(\oo \cdot (\n-\pmb{m})) \ft g(\oo \cdot \pmb{m}).
\end{align*}
	
\noi
Then,
by  Young's inequality, we have
\begin{align}
\|fg\|_{\A_{\oo}(\R)} \leq \|f\|_{\A_{\oo}(\R)} \|g\|_{\A_{\oo}(\R)}.
\label{algebra}
\end{align}

\noi
Namely, $\A_{\oo}(\R)$ is an algebra.
\end{proof}

\subsection{On the linear Schr\"odinger equation}

In this subsection, we study the properties
of solutions to the homogeneous and nonhomogeneous Schr\"odinger equations
in the almost periodic setting.
We  first consider the homogeneous linear Schr\"odinger equation:
\begin{align}
\begin{cases}
i \dt u + \dx^2 u = 0, \\
u|_{t = 0} = f \in \A_{\oo}(\R),
\end{cases}
\quad (t, x) \in \R\times \R.
\label{NLS2}
\end{align}

\noi
Given $f \in \A_{\oo}(\R)$ satisfying \eqref{AP5},
we define
the linear propagator $S(t) = e^{it\dx^2}$ by
\begin{align}
S(t) f
:= \sum_{\n \in \Z^\mathbb{N}} \ft f(\oo\cdot \n)e^{-i (\oo\cdot \n)^2 t} e^{i (\oo \cdot \n)x}.
\label{lin0}
\end{align}

\begin{lemma}\label{LEM:linear1}
Given $f \in \A_{\oo}(\R)$,
let $u = S(t) f$ be as in \eqref{lin0}.
\smallskip

\noi
\textup{(i)} The function $u = S(t) f$ satisfies \eqref{NLS2} in the distributional sense.
Namely, we have
\begin{align}
\iint_{\R\times \R} u \Big( -i \dt \phi + \dx^2 \phi\Big) dx dt
= 0,
\label{lin0a}
\end{align}

\noi
for any test function $\phi\in C^\infty_c(\R_t \times \R_x)$.

\smallskip

\noi
\textup{(ii)}
The function $u = S(t)f$ lies in $C(\R; \A_{\oo}(\R))$.
Moreover, we have
\begin{align}
\|S(t)f\|_{C (\R; \A_{\oo}(\R))}\leq \|f\|_{\A_{\oo}(\R)}.
\label{linear1}
\end{align}

\smallskip

\noi
\textup{(iii)}
The function $u = S(t)f$ is almost periodic in both $t$ and $x$.
Moreover, we have $u \in AP(\R; \A_{\oo}(\R))$.
Namely,
the function $u = S(t)f$ is almost periodic in $t$
with values in $\A_{\oo}(\R)$.

\end{lemma}

In the following, we do not make use of Lemma \ref{LEM:linear1} (iii).
We, however, decided to include it due to its independent interest.
The same comment applies to
Lemma \ref{LEM:linear2} (iii) below.

\begin{proof}
(iii)
Given an enumeration $\{ r_j\}_{j = 1}^\infty$
of  $\Z^{\mathbb N}$,
let $B_N = \{r_j \}_{j = 1}^N$,  $N \in \mathbb N$.
Define $f_N$ as in \eqref{AP6}
and $u_N$ by
\begin{align}
u_N(t)  = S(t) f_N
:= \sum_{\n \in B_N} \ft f(\oo\cdot \n)e^{-i (\oo\cdot \n)^2 t} e^{i (\oo \cdot \n)x}.
\label{lin1}
\end{align}

\noi
By writing  $u_N
= \sum_{\n \in B_N} c_N e^{-i (\oo\cdot \n)^2 t}
$ with $c_N = \ft f(\oo\cdot \n) e^{i (\oo \cdot \n)x} \in \A_{\oo}(\R)$,
we see that $u_N$ is a trigonometric polynomial
with values in the Banach space $\A_{\oo}(\R)$.
Moreover, given $\eps > 0$, there exists $N_0\in \mathbb N$
such that
\begin{align}
\sup_{t \in \R} \| u(t) - u_N(t) \|_{\A_{\oo}(\R)}
= \| f - f_N \|_{\A_{\oo}(\R)}
=  \sum_{\n \in \Z^\mathbb{N} \setminus B_N} |\ft f(\oo\cdot \n)| < \eps
\label{lin1a}
\end{align}

\noi
for all $N \geq N_0$.
Namely, $u = S(t) f$ is uniformly approximated by
the trigonometric polynomials $u_N$.
Therefore, $u$ is almost periodic in $t$ with values in $\A_{\oo}(\R)$.
Since $\A_{\oo}(\R) \subset L^\infty(\R)$,
\eqref{lin1a} implies that
 $u_N$ converges to $u$ in $L^\infty_{t, x}(\R\times \R)$.
Noting that $u_N$ is a trigonometric polynomial
in $t$ and $x$, uniformly converging to $u$,
% for fixed $t$ and in  $t$ for fixed $x$,
the first claim in (iii) follows.

\smallskip
\noi
(i)
Since the sum in \eqref{lin1}
is over a finite set of indices,
we see that $u_N$ is a smooth solution to \eqref{NLS2} in the classical sense,
where the initial condition is replaced by $f_N$.
In particular, we have
\begin{align}
\iint_{\R\times \R} u_N \Big( -i \dt \phi + \dx^2 \phi\Big) dx dt
= 0,
\label{lin2}
\end{align}
	
\noi
for any test function $\phi\in C^\infty_c(\R_t \times \R_x)$.
Then, we have
\begin{align}
\bigg|\iint_{\R\times \R} (u - u_N ) & \Big( -i \dt \phi + \dx^2 \phi\Big) dx dt\bigg| \notag \\
& \leq \|u - u_N \|_{L^\infty_{t, x}(\R\times \R)}
\big( \|\phi\|_{W^{1, 1}_t L^1_x} + \|\phi\|_{L^1_t W^{2, 1}_x}\big)
\too 0,
\label{lin3}
\end{align}

\noi
as $N \to \infty$. Hence, \eqref{lin0a} follows from \eqref{lin2} and \eqref{lin3}.

\smallskip

\noi
(ii)
By (iii), we have $AP(\R; \A_{\oo}(\R))$.
Hence,
the first claim in (ii) follows since an almost periodic function is
continuous.
Lastly, note that \eqref{linear1} follows from \eqref{lin0}
and the definition of the $\A_{\oo}(\R)$-norm.
This completes the proof of Lemma \ref{LEM:linear1}.
\end{proof}

Next, we  consider the nonhomogeneous linear Schr\"odinger equation:
\begin{align}
\begin{cases}
i \dt u + \dx^2 u = F, \\
u|_{t = 0} = 0,
\end{cases}
\quad (t, x) \in \R\times \R,
\label{NLS2a}
\end{align}

\noi
for $F \in L^\infty([-T, T]; \A_{\oo}(\R))$,
$T>0$.
Then, the solution $u$ to \eqref{NLS2a} is formally given by
\begin{align}
u(t) & := -i  \int_0^t S(t - t') F(t') dt' \notag\\
& = -i  \int_0^t  \sum_{\n \in \Z^\mathbb{N}} \ft F(t', \oo\cdot \n)e^{-i (\oo\cdot \n)^2 (t-t')}  e^{i (\oo \cdot \n)x}dt' .
\label{lin7}
\end{align}

\begin{lemma}\label{LEM:linear2}
Let $T>0.$ Given $F \in L^\infty([-T, T]; \A_{\oo}(\R))$,
let $u$ be as in \eqref{lin7}.

\noi
\textup{(i)} The function $u $  satisfies \eqref{NLS2a} in the distributional sense.
Namely, we have
\begin{align}
\iint_{[-T, T]\times \R} u \Big( -i \dt \phi + \dx^2 \phi\Big) dx dt
= \iint_{[-T, T]\times \R} F \phi \, dx dt,
\label{lin8}
\end{align}

\noi
for any test function $\phi\in C^\infty_c([-T, T] \times \R)$.

\noi
\textup{(ii)}
The function $u$ defined in \eqref{lin7} lies in $C([-T, T]; \A_{\oo}(\R))$.\footnote{Indeed, it
follows from the proof that $u$ is uniformly continuous on $[-T, T]$, $T <\infty$,  with values in $ \A_{\oo}(\R)$.}
Moreover,  we have
\begin{align}
\bigg\|  \int_0^t S(t - t') F(t') dt'\bigg\|_{C([-T, T]; \A_{\oo}(\R))}
\le T \|F\|_{L^\infty([-T, T]; \A_{\oo}(\R))}.
\label{linear2}
\end{align}

\noi
\textup{(iii)}
Let $T = \infty$.
Suppose that $F(t)$ is almost periodic in $t$
with values in $\A_{\oo}(\R)$ whose Fourier series is given by
\begin{align}
F(t) \sim \sum_{j = 1}^\infty c_j e^{i \ld_j t}, \qquad c_j \in \A_{\oo}(\R),
\label{F1}
\end{align}

\noi
such that $\{c_j\}_{j = 1}^\infty \in \l^1(\mathbb N; \A_{\oo}(\R))$.
In addition, assume that the following non-resonance condition holds:
\begin{align}
\inf_{j \in \mathbb N} \inf_{\n \in \Z^\mathbb{N}} \big|\ld_j + (\oo\cdot \n)^2\big| > \dl > 0
\label{F2}
\end{align}

\noi
for some $\dl > 0$.	
Then, $u$ defined in \eqref{lin7} is almost periodic in $t$ and $x$.	
Moreover, we have $u \in AP(\R; \A_{\oo}(\R))$.
Namely,
it is almost periodic in $t$
with values in $\A_{\oo}(\R)$.

\end{lemma}

\begin{proof}
(i) Since $F \in  L^\infty([-T, T]; \A_{\oo}(\R))$,
we have
\begin{align}
F(t, x) = \sum_{\n \in \Z^\mathbb{N}} \ft F(t, \oo\cdot \n) e^{i (\oo \cdot \n)x},
\label{lin9}
\end{align}

\noi
for almost every $t \in [-T, T]$.
As before,
given an  enumeration $\{ r_j\}_{j = 1}^\infty$
of  $\Z^{\mathbb N}$,
we define $F_N$ by
\begin{align}
F_N(t, x) = \sum_{\n \in B_N} \ft F(t, \oo\cdot \n) e^{i (\oo \cdot \n)x},
\label{lin9a}
\end{align}

\noi
for $t \in [-T, T]$ such that \eqref{lin9} holds,
where $B_N = \{r_j \}_{j = 1}^N$,  $N \in \mathbb N$.
In the following, by setting $F(t) = 0$ on the exceptional set of measure 0 in $[-T, T]$,
we simply assume that \eqref{lin9}
and \eqref{lin9a} hold for all $t \in [-T, T]$.

Now, define $u_N$ by
\begin{align}
u_N(t) := -i  \int_0^t S(t - t') F_N(t') dt'
= -i  \sum_{\n \in B_N} \int_0^t \ft F(t', \oo\cdot \n)e^{-i (\oo\cdot \n)^2 (t-t')} dt' e^{i (\oo \cdot \n)x}.
\label{lin10}
\end{align}

\noi
Since $\ft F(t, \oo\cdot \n) \in L^\infty([-T, T]) \subset L^1([-T, T])$,
we see that $u_N$ is absolutely continuous in $t$ and smooth in $x$.
Then, it is easy to see that
such $u_N$ satisfies \eqref{NLS2a}
with $F$ replaced by $F_N$ for almost every $t \in [-T, T]$
and every $x \in \R$.
In particular, we have
\begin{align}
\iint_{[-T, T]\times \R} u_N \Big( -i \dt \phi + \dx^2 \phi\Big) dx dt
= \iint_{[-T, T]\times \R} F_N \phi \, dx dt
\label{lin11}
\end{align}

\noi
for any test function $\phi\in C^\infty_c([-T, T] \times \R)$.

Note that $S(t-t') F_N(t')$ is given by
\begin{align*}
S(t-t') F_N(t')  =  \sum_{\n \in B_N} \ft F(t', \oo\cdot \n)e^{-i (\oo\cdot \n)^2 (t-t')}  e^{i (\oo \cdot \n)x}.
\end{align*}

\noi
Then,
we see that $S(t-t') F_N(t')$ converges to $S(t-t') F(t') $ in $\A_{\oo}(\R)$
for each fixed $t'  \in [-T, T]$ (and $t \in \R$).
Moreover, we have
$\sup_N\|S(t-t') F_N(t')\|_{\A_{\oo}(\R)}
\leq \|F(t')\|_{\A_{\oo}(\R)} \in L^1_{t'}([-T, T])$.
Hence, it follows from  Dominated Convergence Theorem that
$u_N(t, x)$ converges to $u(t, x)$ for every $(t, x) \in [-T, T]\times \R$
and  we have
\begin{align}
u(t)
 = -i   \sum_{\n \in \Z^\mathbb{N}} \int_0^t  \ft F(t', \oo\cdot \n)e^{-i (\oo\cdot \n)^2 (t-t')}  dt' e^{i (\oo \cdot \n)x}.
\label{lin11a}
\end{align}

\noi
Also, note that, for each $t\in [-T, T]$,
$F_N(t)$ converges to $F(t)$ in $\A_{\oo}(\R)$.
In particular, $F_N(t, x)$ converges
to $F(t, x)$ for every $(t, x) \in [-T, T]\times \R$.
Moreover, we have
\begin{align*}
|F_N(t, x)| \leq \| F\|_{L^\infty([-T,T]; \A_{\oo}(\R))}
\qquad
\text{and}
\qquad
|u_N(t, x)| \leq T\| F\|_{L^\infty([-T,T]; \A_{\oo}(\R))}
\end{align*}

\noi
for all $(t, x) \in [-T, T] \times \R$.
Therefore,
by Dominated Convergence Theorem applied
to both sides of \eqref{lin11},
we obtain
\eqref{lin8}.

\smallskip

\noi
(ii)
Note that
$ \sum_{\n \in \Z^\mathbb{N} \setminus B_N} |\ft F(t, \oo\cdot \n)|$ converges to 0
for each $t \in [-T, T]$ as $N \to \infty$
and that
$ \sum_{\n \in \Z^\mathbb{N} \setminus B_N} |\ft F(t, \oo\cdot \n)| \leq \| F (t) \|_{\A_{\oo}(\R)}
\in L^1([-T, T])$.
Then, by Dominated Convergence Theorem, we have
\begin{align}
\lim_{N \to \infty} \int_{-T}^T \sum_{\n \in \Z^\mathbb{N} \setminus B_N} |\ft F(t, \oo\cdot \n)| =0.
\label{lin11b}
\end{align}

\noi
Hence, it follows from \eqref{lin10}, \eqref{lin11a}, and \eqref{lin11b}
that $u_N$ converges to $u$ in $L^\infty([-T, T]; \A_{\oo}(\R))$.

Fix $t \in [-T, T]$
and $\eps > 0$.
Then,
there exists
 $N_0 \in \mathbb{N}$ sufficiently large such that
\begin{align}
\|u_{N_0} - u \|_{L^\infty([-T, T]; \A_{\oo}(\R))} < \frac{\eps}{4}.
\label{lin12}
\end{align}

\noi
Also,
from \eqref{lin10}
and Mean Value Theorem,
there exists $\dl_0 >0 $ such that
\begin{align}
\| u_{N_0}(t+\dl)  -  u_{N_0}(t)\|_{\A_{\oo}(\R)}
&   \leq
    \int_t^{t+\dl}\sum_{\n \in B_{N_0}} |\ft F(t', \oo\cdot \n)|dt'  \notag \\
& \hphantom{XX} +
 \int_0^t \sum_{\n \in B_{N_0}} \big| \ft F(t', \oo\cdot \n)
 \big(e^{-i (\oo\cdot \n)^2 (t+\dl)} - e^{-i (\oo\cdot \n)^2 t} \big)\big| dt'
\notag \\
& \leq |\dl|\big( 1+  T \max_{\n \in B_{N_0}} (\oo\cdot\n)^2 \big)\|F\|_{L^\infty([-T, T]; \A_{\oo}(\R))}
< \frac{\eps}{2}\label{lin13}
\end{align}
	
\noi
for all $|\dl| < \dl_0$ such that $t+\dl \in [-T, T]$.
Therefore, from \eqref{lin12} and \eqref{lin13}, we have
\begin{align*}
\| u(t+\dl) - u(t)\|_{\A_{\oo}(\R)}
\leq \| u_{N_0}(t+\dl) - u_{N_0}(t)\|_{\A_{\oo}(\R)} +
\frac{\eps}{2}
 < \eps
\end{align*}

\noi
for all $|\dl| < \dl_0$ such that $t+\dl \in [-T, T]$.
This shows that $u  \in C(\R; \A_{\oo}(\R))$.
Lastly, note that \eqref{linear2} follows from \eqref{lin11a}
and the definition of the $\A_{\oo}(\R)$-norm.

\smallskip

\noi
(iii) By assumption on $c_j$ and Lemma \ref{LEM:AP3},
the Fourier series (in $x$) associated to $c_j$ converges uniformly
to $c_j$ and
we have
\begin{align}
c_j(x) =  \sum_{\n \in \Z^\mathbb{N}} \ft c_j(\oo\cdot \n) e^{i (\oo \cdot \n)x},
\label{F4}
\end{align}

\noi
for all $j \in \mathbb N$.
Given $\eps > 0$, choose $J_0 \in \mathbb{N}$ such that
\begin{align*}
\sup_{t \in \R} \bigg\|\sum_{j = J_0}^\infty c_j e^{i \ld_j t}\bigg\|_{\A_{\oo}(\R)}
\leq \sum_{j = J_0}^\infty \| c_j\|_{\A_{\oo}(\R)} < \eps.
\end{align*}

\noi
In particular, the Fourier series (in $t$) associated to $F$ converges in $\A_{\oo}(\R)$ uniformly in $t$.
Hence, by Lemma \ref{LEM:AP2},
they must converge to $F$.
See also Theorem 6.14 in \cite{C}.
Then, from \eqref{F4},  we have
\begin{align}
F(t, x) = \sum_{j = 1}^\infty c_j(x) e^{i \ld_j t}
= \sum_{j = 1}^\infty \sum_{\n \in \Z^\mathbb{N}} \ft c_j(\oo\cdot \n) e^{i (\oo \cdot \n)x}
e^{i \ld_j t}.
\label{F5}
\end{align}

\noi
Note that the series on the right-hand side of 	\eqref{F5}
converges absolutely and uniformly in $t$ and $x$.
Comparing \eqref{lin9a} and \eqref{F5}, we conclude that
\begin{align}
\ft F(t, \oo\cdot\n)
= \sum_{j = 1}^\infty \ft c_j(\oo\cdot \n)
e^{i \ld_j t}.
\label{F6}
\end{align}

\noi
Then, from \eqref{lin11a} and \eqref{F6} with Dominated Convergence Theorem, we have
\begin{align}
u(t)
&  = -i   \sum_{\n \in \Z^\mathbb{N}}
\int_0^t  \sum_{j = 1}^\infty \ft c_j(\oo\cdot \n)
e^{i (\ld_j + (\oo\cdot \n)^2) t'}  dt' e^{-i (\oo\cdot \n)^2 t}e^{i (\oo \cdot \n)x} \notag \\
&  =    \sum_{\n \in \Z^\mathbb{N}}
 \sum_{j = 1}^\infty \ft c_j(\oo\cdot \n)
\frac{e^{i  (\oo\cdot \n)^2 t} - e^{i \ld_j  t} }{\ld_j + (\oo\cdot \n)^2}
e^{i (\oo \cdot \n)x}.
\label{F7}
\end{align}

\noi
In view of the non-resonance condition \eqref{F2},
the double series on the right-hand side of \eqref{F7}
converges absolutely and uniformly in $t$ and $x$.
This proves almost periodicity of $u$ in both $t$ and $x$.
Moreover, noting that
$\ft c_j(\oo\cdot \n) e^{i (\oo \cdot \n)x} \in \A_{\oo}(\R)$
with the norm given by $|\ft c_j(\oo\cdot \n) |$
for each $j \in \mathbb{N}$ and $\n \in \Z^\mathbb{N}$,
we see that $u$ is a limit of $\A_{\oo}(\R)$-valued
trigonometric polynomials, uniformly in $t$,
i.e.~$u$ has the approximation property.
Therefore,
$u$ is almost periodic in $t$ with values in
$\A_{\oo}(\R)$.
\end{proof}

\subsection{Fixed point argument}

Finally, we are ready to prove Theorem \ref{THM:LWP}.
Given $f \in \A_{\oo}(\R)$, define $\G = \G_f$ by
\begin{align}
\G u(t) = \G_f u (t): = S(t) f -i \int_0^t S(t - t') \N(u)(t') dt'.
\label{NLS4}
\end{align}

\noi
We show that
 $\G$ is a contraction on
\[ B = \big\{ u \in C([-T, T]; \A_{\oo}(\R)):\,  \|u\|_{C([-T, T]; \A_{\oo}(\R))} \leq 2 \|f \|_{\A_{\oo}(\R)}\big\}\]

\noi
for some $T>0.$
Indeed,
by \eqref{algebra} and Lemmata \ref{LEM:linear1} and \ref{LEM:linear2},
we have
\begin{align*}
\|u\|_{C([-T, T]; \A_{\oo}(\R))}
\leq \| f\|_{\A_{\oo}(\R)} + T \|u\|^p_{C([-T, T]; \A_{\oo}(\R))},
\end{align*}

\noi
and
\begin{align*}
\|u - v \|_{C([-T, T]; \A_{\oo}(\R))}
\leq  C T \Big(\|u\|^{p-1}_{C([0, T]; \A_{\oo}(\R))}
+ \|v\|^{p-1}_{C([-T, T]; \A_{\oo}(\R))}\Big)
\|u-v\|_{C([-T, T]; \A_{\oo}(\R))}
\end{align*}

\noi
for $u, v \in B$.
Therefore, we conclude that
$\G$ is a contraction on $B$
as long as
\[ T < \frac{1}{ 2^{p+1} C\|f\|^{p-1}_{\A_{\oo}(\R)}}.\]

\noi
The  Lipschitz continuity of the solution map on $\A_{\oo}(\R)$
follows from a similar argument.

\section{Finite time blowup solutions}
\label{SEC:blowup}

In this section, we present the proof of Theorem \ref{THM:blowup}.
We only work on the positive time intervals $[0, T)$, $0< T \leq \infty$,
and prove Theorem \ref{THM:blowup} (i) under \eqref{sign1}.
Note that Theorem \ref{THM:blowup} (ii) easily follows
from Theorem \ref{THM:blowup} (i).
Given a solution  $u$  to \eqref{ZNLS1} on $(-T, 0]\times \R$,
the function   $v(t, x) := u(-t, x)$  satisfies $i \dt v - \dx^2 v = (-\ld) |v|^p$
 on  $[0, T)\times \R$.
Then, it suffices to  note that
 the sign change in front of  $\dx^2 v$ does not affect the proof
 of Theorem \ref{THM:blowup} (i).
In the following,
fix $p\in 2\mathbb{N}$.
and $\oo = \{\o_j\}_{j = 1}^\infty \in \R^{\mathbb N}$.

First, let us define the notion of weak solutions as in \cite{IW, O}.
\begin{definition} \label{DEF1} \rm
Let $T> 0$. We say that $u$ is a local  weak solution to \eqref{ZNLS1}
on $[0, T)$ with initial condition $u|_{t = 0} = f \in AP(\R)$
if  $u \in L^\infty([0, T); \A_{\oo}(\R))$ and
\begin{align} \label{ZNLS2}
\int_0^T \int_{\R} u\Big(-i \dt \phi + \dx^2 \phi\Big) dx dt
= i \int_{\R} f(x) \phi( 0, x) dx
+ \ld \int_0^T \int_{\R} |u|^p \phi \, dx dt,
\end{align}

\noi
for any test function $\phi \in C^\infty_c((-\infty, T)\times \R)$.
If $T>0$ can be made arbitrarily large,
then we say that $u$ is a global weak solution on $[0, \infty)$.
\end{definition}

We now present two
important  propositions for proving Theorem \ref{THM:blowup}.

\begin{proposition} \label{PROP:blowup}
Assume \eqref{sign1}.
If $u$ is a global-in-time weak solution to  \eqref{ZNLS1} on $[0, \infty)$,
then $u(t) = 0$ for almost every $t \in [0, \infty)$.
\end{proposition}

Our solution constructed in Theorem \ref{THM:LWP}
satisfies
 the following Duhamel formulation:
\begin{align}
\label{ZNLS3}
u(t) = S(t) f -i  \ld \int_0^t S(t-t') |u(t')|^p dt'.
\end{align}

\noi
The next proposition guarantees
that our solution $u$ in Theorem \ref{THM:LWP}
indeed satisfies
the weak formulation \eqref{ZNLS2}.

\begin{proposition} \label{PROP1}
Given $ f \in \A_{\oo}(\R)$,
 if $u\in C([0, T);\A_{\oo}(\R))$ satisfies the Duhamel formulation \eqref{ZNLS3}
on $[0, T)$ for some  $T>0$, then it is a weak solution to \eqref{ZNLS1} on $[0, T)$ in the sense of Definition \ref{DEF1}.
\end{proposition}

We first prove Theorem \ref{THM:blowup}
using Propositions \ref{PROP:blowup} and \ref{PROP1}.
Then, we present the proofs of Propositions \ref{PROP:blowup} and \ref{PROP1}.

Let $u \in C([0, T_+); \A_{\oo}(\R))$ be the solution to \eqref{ZNLS1}
with $u|_{t = 0} = f \in \A_{\oo}(\R)$, where $[0, T_+)$ is the forward maximal time interval
of existence.
Suppose that \eqref{sign1} holds.
On the one hand, this implies $f\ne 0$.
Then, it follows from
 uniqueness in Theorem \ref{THM:LWP} that
 $u \not\equiv 0$.
On the other hand,
from Proposition \ref{PROP:blowup}
and the continuity of $u$ in time, we conclude that  $u \equiv 0$.
This is a contradiction.
Therefore,
the  solution $u$ can not be global on $[0, \infty)$
and must blow up at some finite time $T_+>0$.

Theorem \ref{THM:LWP}
guarantees  the following blowup alternative;
if $u$ is a solution in $C([0, T);\A_{\oo}(\R))$, then
either (a) there exists $\eps > 0$ such that $u$ can be extended
to $[0, T+\eps)$ or (b) $\liminf_{t \nearrow T} \|u(t)\|_{\A_{\oo}(\R)} = \infty.$
Since $T_+ < \infty$,
we conclude that
 $\liminf_{t \nearrow T_+} \|u(t)\|_{\A_{\oo}(\R)} = \infty.$
 This completes the proof of Theorem \ref{THM:blowup}.

In the remaining part of this section, we present the proofs of
Propositions \ref{PROP:blowup} and  Proposition \ref{PROP1}.

\begin{proof}[Proof of Proposition \ref{PROP:blowup}]
The proof is based on the test-function method by
Zhang \cite{Z}.
While we closely follows the arguments in \cite{IW, O},
some care must be taken due to the almost periodic nature of the problem.
In particular, while there is only one parameter for the space-time cutoff functions in  \cite{IW, O},
we need to introduce a space-time cutoff function depending on two parameters.
This is mainly due to the fact that the $L^p$-norm
of a non-zero almost periodic function
 is infinite for any finite value of $p$
 and thus   we need to use the mean value $M(|u|^p)$ of $|u|^p$ instead.
For simplicity of presentation, we only prove Proposition \ref{PROP:blowup}
when $\Re \ld \cdot  \Im M(f) < 0$.
Without loss of generality, assume
$\text{Re}\, \ld > 0$ and
 $\Im M(f)  < 0$.

Let $\theta: [0, \infty)\to  [0, 1]$
be a smooth  cutoff function
 supported on  $[0,1)$
 such that
$\theta \equiv 1$ on $[0, \frac{1}{2})$
and $|\theta'(t)|, |\theta''(t)| \leq C$
for all $t \in [0, \infty)$.
Moreover, we impose that
\begin{align}
\frac{|\theta' (t)|^2}{\theta(t)} \leq C
\label{P0}
\end{align}

\noi
for all $t \in [0, 1)$.
Fix  $\eps \in (0, 1)$.
Given $L > 1$,
we define a smooth cutoff function
 $\eta_L:\R\to [0, 1]$
by
\[ \eta_L(x)
= \begin{cases}
1 & \text{for } |x|\leq  L-L^{1-\eps}, \\
\displaystyle \theta\bigg(\frac{|x| -  L+L^{1-\eps}}{L^{1-\eps}}\bigg)
& \text{for } |x|>  L-L^{1-\eps}.
\end{cases}
\]

\noi
In particular,  $\supp \eta_L \subset [-L, L]$
and by \eqref{P0}, we have
\begin{align}
|\dx^2 \eta_L(x)|, \  \frac{|\dx \eta_L (x)|^2}{\eta_L(x)} \leq \frac{C}{L^{2-2\eps}}
\label{P0b}
\end{align}

\noi
on  $(-L, L)$.
Finally, given $R, L > 1$, we define the space-time cutoff function $\phi_{R, L}$
by  $\phi_{R, L}(t, x) = \theta_R(t)\cdot \eta_L(x)$,
where   $\theta_R(t) := \theta(R^{-2}t)$.

For $R, L > 1$, define $\I_{R, L}$ by
\[\I_{R, L} : = \Re \ld \cdot \frac{1}{2L} \int_0^{R^2}  \int_{-L}^L |u |^p \phi_{R, L}^{p'} \, dx dt,
\]

\noi
where $p'$ is the H\"older conjugate exponent of $p$.
Since  $\Im M(f)  < 0$ and $AP(\R)\subset L^\infty(\R)$, we have
\begin{align*}
\lim_{L \to \infty}
& \Im\bigg[ \frac{1}{2L}\int_{-L}^L f(x)\eta_L^{p'}( x) dx\bigg] \\
& \leq
\lim_{L \to \infty}
\frac{L-L^{1-\eps}}{L}
\Im\bigg[ \frac{1}{2(L-L^{1-\eps})}\int_{-L+L^{1-\eps}}^{L-L^{1-\eps}} f(x) dx\bigg]
+ \lim_{L \to \infty}  \frac{C}{L^\eps} \|f\|_{L^\infty(\R)}\\
& =  \Im M(f)  < 0.
\end{align*}

\noi
Hence,
there exists $L_0 > 1$ such that we have
\begin{align}
\Im\bigg[ \frac{1}{2L}\int_{-L}^L f(x)\phi_{R, L}^{p'}(0,  x) dx\bigg] < 0
\label{P0a}
\end{align}
	
\noi
for all $L \geq L_0$.

Let $T > R^2$.
Then, from the weak formulation  \eqref{ZNLS2} with
\eqref{P0a}, \eqref{P0b},  and H\"older's inequality, we have
\begin{align}
\I_{R, L}
& =  \Im  \bigg[ \frac{1}{2L} \int_{-L}^L f(x) \phi_{R, L}^{p'}(0, x) dx\bigg]
+ \Re \bigg[\frac{1}{2L} \int_0^{R^2}  \int_{-L}^L u
\Big( -i \dt (\phi_{R, L}^{p'}) + \dx^2 (\phi_{R, L}^{p'}) \Big) dx dt\bigg] \notag \\
& <
\frac{1}{2L} \int_0^{R^2}  \int_{-L}^L \big|u\cdot \dt (\phi_{R, L}^{p'})\big|  dx dt
+\frac{1}{2L} \int_0^{R^2}  \int_{-L}^L \big|u\cdot\dx^2 (\phi_{R, L}^{p'})\big|  dx dt\notag \\
& \les
\frac{1}{R^2}\cdot   \frac{1}{2L}   \int_{\frac{R^2}{2}}^{R^2}  \int_{-L}^L |u (t, x) |
\eta_L (x)\phi_{R, L}^{p'-1} (t, x) \Big|\dt \theta \Big(\frac{t}{R^2}\Big)\Big|
 dx dt \notag \\
& \hphantom{XXXXXX} +
     \frac{1}{2L}   \int_{\frac{R^2}{2}}^{R^2}  \int_{-L}^L |u (t, x) |
\theta_R(t) \phi_L^{p'-1} (t, x)
 \bigg(
 \frac{|\dx \eta_L(x)|^2 }{\eta_L( x)}+ |\dx^2 \eta_L( x)|\bigg)  dx dt \notag \\
& \lesssim
\bigg(\frac{1}{R^2}+ \frac{1}{L^{2-2\eps}}\bigg) \bigg(\frac{1}{2L}
\int_{\frac{R^2}{2}}^{R^2} \int_{-L}^L 1 \, dx dt\bigg)^\frac{1}{p'}
\bigg(  \frac{1}{2L}  \int_{\frac{R^2}{2}}^{R^2}  \int_{-L}^L |u (x, t) |^p   \phi_{R, L}^{p'}(x, t)
 dx dt\bigg)^\frac{1}{p} \notag \\
& \lesssim \big(R^{-\frac{2}{p}} + R^\frac{2}{p'} L^{-2+ 2\eps}) \I_{R,  L}^\frac{1}{p}
\label{P0e}
\end{align}

\noi
for all $L \geq L_0$.
Hence,  noting that $p > 1$ and $\eps \in (0, 1)$, we have
\begin{equation*}
\I_{R, L} \lesssim R^{-\frac{2}{p-1}} + R^2 L^{-\frac{2p}{p-1}(1-\eps)} \leq C < \infty,
\end{equation*}

\noi
as long as
\begin{align}
L \gg  \max \big(R^{\frac{ p-1}{(1-\eps)p}},  L_0\big)
\quad \text{and} \quad R> 1.
\label{P0c}
\end{align}

\noi
Since $\phi_{R, L}(t, x)\equiv 1$ on $[0, \frac{R^2}{2})\times [-\frac{L}{2}, \frac{L}{2}]$
for $L\gg1$,
we have
\begin{equation}
  \int_0^{\frac{R^2}{2}}  \frac{1}{2L} \int_{-\frac{L}{2}}^{\frac{L}{2}} |u|^p   dx dt
  \lesssim R^{-\frac{2}{p-1}} + R^2 L^{-\frac{2p}{p-1}(1-\eps)} \leq C < \infty,
\label{P0d}
\end{equation}

\noi
independent of $L, R\gg 1$, satisfying \eqref{P0c}.
Since  $u \in L^\infty([0, T); \A_{\oo}(\R))$, we have
\begin{equation*}
\frac{1}{2L} \int_{-\frac L 2}^{\frac L2} |u(t) |^{p} dx \leq \frac 12 \|u\|^p_{L^\infty([0, T); \A_{\oo}(\R))}
\end{equation*}

\noi
for all $L > 0$ and $t \in [0, T)$.
Then, by Dominated Convergence Theorem, we have
\begin{align*}
\lim_{L\to \infty}
\int_0^{\frac{R^2}{2}}  \frac{1}{2L} \int_{-\frac{L}{2}}^{ \frac{L}{2}} |u(t)|^p   dx dt
= \int_0^{\frac{R^2}{2}}  \lim_{L\to \infty}
\frac{1}{2L} \int_{-\frac{L}{2}}^{ \frac{L}{2}} |u(t)|^p   dx dt
= \frac{1}{2} \int_0^{\frac{R^2}{2}}
M\big(|u(t)|^p\big) dt
\end{align*}

\noi
for every fixed $R>1$.
Hence, by Monotone Convergence Theorem with \eqref{P0d}, we obtain
\begin{align*}
 \int_0^\infty
M\big(|u(t)|^p\big) dt
& =
\lim_{R\to \infty} \int_0^{\frac{R^2}{2}}
M\big(|u(t)|^p\big) dt\\
& =
\lim_{R\to \infty}\lim_{L\to \infty}
\int_0^{\frac{R^2}{2}}  \frac{1}{L} \int_{-\frac{L}{2}}^{ \frac{L}{2}} |u(t)|^p   dx dt
= 0.
\end{align*}

\noi
Therefore, by Lemma \ref{LEM:AP1},
we conclude that $u(t) = 0$ for almost every $t$.
\end{proof}

\begin{remark}\label{REM:bound2}\rm
Suppose $\Re \ld > 0$ and $\Im M(f) > 0$.
Let $u$ be a global weak solution to \eqref{ZNLS1} on $[0, \infty)$
with $u|_{t = 0} = f$.
Then, by repeating the computation in \eqref{P0e}
and taking the limits of both sides as $L\to \infty$
with Dominated Convergence Theorem as above, we
\begin{align*}
\I_{R}
\leq \Im M(f) + C R^{-\frac{2}{p}} \I_{R}^\frac{1}{p}.
%\label{P0e}
\end{align*}

\noi
where
\[\I_{R} : = \Re \ld  \int_0^{R^2}  M( |u (t) |^p) \theta_{R}^{p'}(t) \, dt.
\]

\noi
Then, by the continuity argument and Fatou's lemma, we obtain
\begin{align} \int_0^\infty
M\big(|u(t)|^p\big) dt <\infty.
\label{decay1}
\end{align}

\noi
Hence, it follows from   Lemma \ref{LEM:AP1} that
 any global solution on $[0, \infty)$ must
go to 0 as $t\to \infty$ in some averaged sense.
In view of Proposition \ref{PROP1},
the same conclusion holds for a global solution $u \in C([0, \infty); \A_{\oo}(\R))$
satisfying the Duhamel formulation \eqref{ZNLS3}.

\end{remark}

Finally,  we present the proof of Proposition \ref{PROP1}.
\begin{proof}[Proof of Proposition \ref{PROP1}]
Let $u$ be  a  solution in $ C([0, T);\A_{\oo}(\R))$,
satisfying the Duhamel formulation \eqref{ZNLS3}.
Write $u(t) = S(t) f +\D(u)(t)$,
where $\mathcal{D}(u)$ is given by
\[\mathcal{D}(u)(t) = -i   \int_0^t S(t-t') \ld |u(t')|^p dt'.\]

\noi
In the following, $\phi$ denotes a
 test function in $ C^\infty_c((-\infty, T)\times \R)$.
We use $W^{s, p}_T$ and $L^p_T$
to denote
$W^{s, p}_t([0, T])$ and $L^p_t([0, T])$, respectively.

First, we show that the linear part $S(t) f$ satisfies
\begin{align}\label{L1}
\int_0^T \int_{\R} S(t)f \Big(-i \dt \phi + \dx^2 \phi\Big) dx dt
= i \int_{\R} f(x) \phi( 0, x) dx.
\end{align}

\noi
Given $f\in \A_{\oo}(\R)$,
define $f_N$ by \eqref{AP6}.
Then, the corresponding linear solution $S(t)f_N$
is given by \eqref{lin1}.
By integration by parts, we have
\begin{align}
\int_0^T \int_{\R}  S(t) f_{N}\Big( & -i \dt   \phi + \dx^2 \phi\Big) dx dt \notag\\
&  =
\sum_{\n \in B_N} \ft f(\oo\cdot \n)
\int_0^T \int_{\R}
e^{-i (\oo\cdot \n)^2 t} e^{i (\oo \cdot \n)x}
\Big(-i \dt \phi + \dx^2 \phi\Big) dx dt \notag \\
& = i 
\sum_{\n \in B_N} \ft f(\oo\cdot \n)
 \int_{\R}
 e^{i (\oo \cdot \n)x}
\phi(0, x) dx
= i \int_{\R} f_N(x) \phi( 0, x) dx. \label{L2}
\end{align}

\noi
Recall that $f_N$ converges to $f$ in $\A_{\oo}(\R)$ and in $L^\infty(\R)$.
Then, by H\"older's inequality and Lemma \ref{LEM:linear1} (ii), we have
\begin{align}
\bigg|\int_0^T \int_{\R} \big(S(t) f &  - S(t) f_N\big)\Big(-i \dt \phi + \dx^2 \phi\Big)\, dx dt\bigg| \notag\\
&  \leq \| S(t)(f -f_N)\|_{L^\infty_{t, x }}
\big( \|\phi\|_{W^{1, 1}_T L^1_x} + \|\phi\|_{L^1_T W^{2, 1}_x}\big)\notag\\
& \leq \| f -f_N\|_{\A_{\oo}(\R)}
\big( \|\phi\|_{W^{1, 1}_T L^1_x} + \|\phi\|_{L^1_T W^{2, 1}_x}\big)
\too 0,  \label{L3}
\end{align}

\noi
as $N \to \infty$.
Similarly, we have
\begin{align}
\bigg| \int_{\R} \big(f (x)- f_N(x)\big) \phi( 0, x) dx\bigg|
 \leq \| f -f_N\|_{L^\infty}\|\phi(0, x)\|_{L^1_x} \too 0,
\label{L4}
\end{align}

\noi
as $N \to \infty$.
Hence, \eqref{L1} follows from \eqref{L2}, \eqref{L3}, and \eqref{L4}.

Since $u \in C([0, T); \A_{\oo}(\R))$,
we have $\ld |u|^p \in C([0, T); \A_{\oo}(\R))$.
Given an enumeration $\{ r_j\}_{j = 1}^\infty$
of  $\Z^{\mathbb N}$,
define $\D_N$ by
\begin{align*}
\D_N(t, x)
= -i  \ld \sum_{\n \in B_N} \int_0^t \ft{|u|^p}(t', \oo\cdot \n)e^{-i (\oo\cdot \n)^2 (t-t')} dt' e^{i (\oo \cdot \n)x}
\end{align*}

\noi
for $t \in [0, T]$, where $B_N = \{r_j \}_{j = 1}^N$,  $N \in \mathbb N$.
Since $\D_N$ is a finite linear combination of smooth functions,
noting that $\D_N(0, x) = \phi(T, x) = \dx \phi(T, x) = 0$ for all $x \in \R$,
 integration by parts  yields
\begin{align*}
 \int_0^T \int_{\R} \D_N \Big(-i \dt \phi + \dx^2 \phi\Big)\, dx dt
= \ld \int_0^T \int_{\R} U_N \phi \, dx dt,
\end{align*}

\noi
where $U_N$ is defined by
\begin{align*}
U_N(t, x) =  \sum_{\n \in B_N}
\ft {|u|^p}(t, \oo \cdot \n)
 e^{i (\oo \cdot \n)x}.
 \end{align*}

\noi
Then, by proceeding as in  the proof of Lemma \ref{LEM:linear2} (i),
we obtain
\begin{align}
 \int_0^T \int_{\R} \D(u) \Big(-i \dt \phi + \dx^2 \phi\Big)\, dx dt
= \ld \int_0^T \int_{\R} |u|^p \phi \, dx dt.
\label{N4}
\end{align}

\noi
The identity \eqref{ZNLS2} follows from \eqref{L1} and  \eqref{N4}.
 \end{proof}

\end{document}